\documentclass[11pt]{article}

\date{}
\usepackage{verbatim}
\usepackage[utf8]{inputenc}
\usepackage[T1]{fontenc}
\usepackage{amsthm}
\usepackage{amssymb}
\usepackage{amsmath}
\usepackage{enumerate}
\usepackage{subcaption}
\usepackage[tikz]{bclogo}

\usepackage{tkz-graph}
\usepackage{tikz}
\usetikzlibrary{decorations.pathreplacing}
\usetikzlibrary{decorations.pathmorphing}
\usetikzlibrary{arrows}

\usepackage{algorithm}
\usepackage{algpseudocode}

\tikzstyle{b_vertex}=[circle,fill=black!50,text=white,inner sep=0.8mm,draw]
\tikzstyle{w_vertex}=[circle,fill=white!100,inner sep=0.8mm,draw]
\tikzstyle{d_vertex}=[circle,fill=white!100,inner sep=0.8mm,draw,densely dashed]
\tikzstyle{s_vertex}=[circle,fill=white!100,inner sep=0.8mm,draw,very thick]
\tikzstyle{sd_vertex}=[circle,fill=white!100,inner sep=0.8mm,draw,very thick,densely dashed]
\tikzstyle{point}=[circle,fill=black,inner sep=0.3mm]
\tikzstyle{path_edge}=[thick]
\tikzstyle{snake_line}=[decoration={snake, amplitude = .6mm, segment length = 2mm},decorate,color=gray]
\tikzstyle{strong_line}=[line width=0.25mm, black ]
\tikzstyle{rbox}=[rounded corners=15pt, color=gray!50]
\tikzstyle{strong_arrow}=[arrows={-latex}, line width=1.5mm, color=gray!30]

\newtheorem{theorem}{Theorem}
\newtheorem{lemma}{Lemma}
\newtheorem{corollary}{Corollary}

\newtheorem{proposition}{Proposition}
\newtheorem{conjecture}{Conjecture}

\newcommand{\decomp}{chain~}

\author{Vadim Lozin\thanks{Mathematics Institute, University of Warwick, Coventry, CV4 7AL. Email: V.Lozin@warwick.ac.uk.}  
\and Viktor Zamaraev\thanks{Mathematics Institute, University of Warwick, Coventry, CV4 7AL. Email: V.Zamaraev@warwick.ac.uk.}}

\title{The structure and the number of $P_7$-free bipartite graphs\thanks{This research was supported by EPSRC, grant EP/L020408/1.}}

\begin{document}
\maketitle

\begin{abstract}
We show that the number of labelled $P_7$-free bipartite graphs with $n$ vertices grows as $n^{\Theta(n)}$.
This resolves an open problem posed by Allen \cite{Peter}, and
completes the description of speeds of monogenic classes of bipartite graphs.
Our solution is based on a new decomposition scheme of bipartite graphs,
which is of independent interest.
\end{abstract}
\section{Introduction}

A {\it graph property} is an infinite class of graphs closed under isomorphism. 
A property is {\it hereditary} if it is closed under taking induced subgraphs.
The number of $n$-vertex labelled graphs in a property $X$ is known as 
the {\it speed} of $X$ and is denoted by $X_n$.

According to Ramsey's Theorem, there exist precisely two minimal hereditary properties: the complete graphs and the edgeless graphs. 
In both cases, the speed is obviously $X_n=1$. On the other extreme, lies the set of all simple graphs, in which case the speed is $X_n=2^{\binom{n}{2}}$.
Between these two extremes, there are uncountably many other hereditary properties and their speeds have been extensively studied, 
originally in the special case of a single forbidden subgraph, and more recently in general. For example,
Erd\H os et al. \cite{EKR76} and Kolaitis et al. \cite{KPR87} studied $K_r$-free graphs, Erd{\H{o}}s et al. \cite{EFR86} studied properties where
a single graph is forbidden as a subgraph (not necessarily induced), and Pr\"omel and Steger obtained a
number of results \cite{PS1,PS2,PS3} for properties defined by a single forbidden induced subgraph. This line
of research culminated in a breakthrough result stating that for every hereditary property $X$ different from
the set of all finite graphs,
\begin{equation}
\lim_{n\to\infty}\frac{\log_2 X_n}{\binom{n}{2}}=1-\frac{1}{k(X)},
\end{equation}
where $k(X)$ is a natural number called the {\it index} of $X$. To define this notion, let us denote by ${\cal E}_{i,j}$ the class
of graphs whose vertices can be partitioned into at most $i$ independent sets and $j$ cliques. In particular,
${\cal E}_{2,0}$ is the class of bipartite graphs, ${\cal E}_{0,2}$ is the class of co-bipartite (i.e. complements of bipartite) graphs and ${\cal E}_{1,1}$ is the class of split graphs. 
Then $k(X)$ is the largest $k$ such that $X$ contains ${\cal E}_{i,j}$ with $i+j = k$. This result was obtained independently by Alekseev \cite{Ale92} and Bollob\' as
and Thomason \cite{BT95,BT97} and is known nowadays as the Alekseev-Bollob\'as-Thomason Theorem (see e.g. \cite{almost}). 
%
This theorem characterizes hereditary properties of high speed, i.e. properties of index $k(X)>1$. 
The asymptotic structure of these properties was studied in \cite{almost}.
For properties of index 1, known as {\it unitary classes}, these results are useless, which is unfortunate, because 
the family of unitary classes contains a variety of properties of theoretical or practical importance, such as
line graphs, interval graphs, permutation graphs, threshold graphs, forests, planar graphs
and, even more generally, all proper minor-closed graph classes \cite{small}, all classes of graphs
of bounded vertex degree, of bounded tree- and clique-width \cite{Allen}, etc.

A systematic study of hereditary properties of low speed was initiated by Scheinerman and Zito in \cite{low}. 
In particular, they distinguished the first four lower layers in the family of unitary classes:
constant (classes $X$ with $X_n = \Theta(1)$), polynomial ($X_n = n^{\Theta(1)})$, 
exponential ($X_n = 2^{\Theta(n)})$ and factorial ($X_n = n^{\Theta(n)})$. 
Independently, similar results have been obtained by Alekseev in \cite{Ale97}.
Moreover, Alekseev described the set of minimal classes in all the four lower layers and the 
asymptotic structure of properties in the first three of them. 
A more detailed description of the polynomial and exponential layers was obtained by Balogh, Bollob\'as and Weinreich in \cite{BBW00}.
However, the factorial layer remains largely unexplored and the asymptotic structure is known only 
for properties at the bottom of this layer, below the Bell numbers \cite{BBW00, Bell}. 
On the other hand, the factorial properties constitute the core of the unitary family, as all the interesting classes mentioned above (and many others) are factorial.
To simplify the study of the factorial properties, we proposed in  \cite{conjecture} the following conjecture.

\medskip
\noindent
{\bf Conjecture on factorial properties}. {\it A hereditary graph property $X$ is factorial if and only if the fastest of the following three
properties is factorial: bipartite graphs in $X$, co-bipartite graphs in $X$, split graphs in $X$}.

\medskip
To justify this conjecture we observe that if in the text of the conjecture we replace the word ``factorial'' by any of the
lower layers (constant, polynomial or exponential), then the text becomes a valid statement. Also, the ``only if'' part of the
conjecture is true, because all minimal factorial classes are subclasses of bipartite, co-bipartite and split graphs. 
Finally, in \cite{conjecture} this conjecture was verified for all hereditary classes defined by forbidden induced subgraphs with at most 4 vertices.

The above conjecture reduces the question of characterizing the factorial layer from the family of all hereditary
properties to those which are bipartite, co-bipartite and split. Taking into account the obvious relationship between
bipartite, co-bipartite and split graphs, this question can be further reduced to hereditary properties of bipartite graphs only.
However, even with this restriction a full characterization of the factorial layer remains a challenging open
problem. In \cite{Peter}, Allen studied this problem for classes of bipartite graphs defined by a single forbidden induced {\it bipartite} subgraph. 
We call such classes monogenic.
Allen characterized all monogenic classes of bipartite graphs according to their speeds, with one exception: the class of $P_7$-free bipartite graphs.
In the present paper, we complete this characterization by showing that  the class of $P_7$-free bipartite graphs is factorial. 

We prove the main result of the paper in two steps. First, in Section~\ref{sec:scheme} we introduce a decomposition scheme that provides a factorial upper bound, 
and then in Section~\ref{sec:P7} we apply this scheme to $P_7$-free bipartite graphs. All preliminary information related to the topic of the paper can be found 
in Section~\ref{sec:prel}.

\section{Preliminaries}
\label{sec:prel}
We study simple labelled graphs, i.e. undirected graphs without loops and multiple edges with vertex set $\{1,2,\ldots,n\}$ for some natural $n$. 
The vertex set and the edge set of a graph $G$ is denoted by $V(G)$ and $E(G)$, respectively.
As usual, $P_n$ is a chordless path and $K_n$ is a complete graph with $n$ vertices.  
For a subset $A \subseteq V(G)$, we denote by $N_G(A)$ the neighbourhood of $A$ in $G$, i.e. the set of vertices of $G$ outside of $A$ 
that have at least one neighbour in $A$. If $A=\{a\}$, we write $N_G(a)$ to simplify the notation.  
A vertex $x \in V(G) \setminus A$ \textit{is complete to} $A$ if $A\subseteq N_G(x)$, and $x$
\textit{is anticomplete to} $A$ if $A\cap N_G(x)=\emptyset$. Similarly, two disjoint subsets $A$ and $B$ of $V(G)$ are {\it complete} to each other if 
every vertex of $B$ is complete to $A$, and they are {\it anticomplete} to each other if 
every vertex of $B$ is anticomplete to $A$.

We denote the union of two disjoint graphs $G$ and $H$ by $G+H$. The join of $G$ and $H$ is obtained from  $G+H$ by adding all possible edges between $G$ and $H$.

In a graph, an {\it independent set} is a subset of pairwise non-adjacent vertices, and a {\it clique} is a subset of pairwise adjacent vertices.
A {\it clique cutset} in a connected graph is a clique whose removal disconnects the graph. 

A graph is {\it bipartite} if its vertex set can be partitioned into at most two independent sets. 
When we talk about bipartite graphs, we assume that each graph is given together with a bipartition of its vertex set
into two parts (independent sets), say left and right, and we denote a bipartite graph with parts $U$ and $W$ by $G = (U, W, E)$,
where $E$ stands for the set of edges. The bipartite complement of a bipartite graph $G = (U, W, E)$ is the bipartite
graph $\overline{G^b} = (U, W, E')$, where two vertices $u \in U$ and $w \in W$ are adjacent in $G$ if and only if they are non-adjacent in $\overline{G^b}$.

Given a subset $A \subseteq V(G)$, we denote by $G[A]$ the subgraph of $G$ induced by $A$. 
If a graph $G$ does not contain an induced subgraph isomorphic to a graph $H$, then we say that $G$ is $H$-free and call $H$ a forbidden induced subgraph for $G$.
It is well known that a class of graphs is hereditary if and only if it can be characterized by means of minimal forbidden induced subgraphs. 
In the present paper, we study bipartite graphs which are $P_7$-free. Observe that the bipartite complement of a $P_7$ is a $P_7$ again, and hence 
the bipartite complement of a $P_7$-free bipartite graph is also $P_7$-free.


\section{Chain decomposition of bipartite graphs}
\label{sec:scheme}

In this section, we introduce a decomposition scheme generalizing some of the previously known decompositions of bipartite graphs, such as canonical decomposition \cite{canonical}.
We call our decomposition {\it chain decomposition} and formally define it in Section~\ref{sec:decomp}. Then in Section~\ref{sec:dec_scheme} we describe 
the scheme, i.e. a decomposition tree based on chain decomposition. This scheme provides a factorial representation 
for a large class of bipartite graphs, which we call {\it chain-decomposable}. We prove a factorial upper bound for this class in Section~\ref{sec:numOfGraphs}.

\subsection{Chain decomposition}
\label{sec:decomp}

Let $G=(U,W,E)$ be a bipartite graph and $k>0$ a natural number. We say that $G$ admits a $k$-\textit{\decomp decomposition} if
$U$ can be partitioned into subsets $A_1, \ldots, A_k, C_1, \ldots, C_k$ and $W$ can be partitioned into subsets $B_1, \ldots, B_k, D_1, \ldots, D_k$
in such a way that
\begin{itemize}
	\item for every $i \leq k-1$, the sets $A_i, B_i,C_i,D_i$ are non-empty. For $i=k$, at least one of the sets $A_i, B_i,C_i,D_i$ must be non-empty.
	\item for each $i =1, \ldots, k$, 
\begin{itemize}
\item every vertex of $B_i$ has a neighbour in $A_i$; 
\item every vertex of $D_i$ has a neighbour in $C_i$;
\end{itemize}
	\item for each $i = 2, \ldots, k-1$, 
\begin{itemize}
\item every vertex of  $A_i$ has a non-neighbour in $B_{i-1}$;
\item	every vertex of $C_i$ has a non-neighbour in $D_{i-1}$;
\end{itemize}		
	\item for each $i = 1, \ldots, k$, 
\begin{itemize}
\item the set $A_i$ is anticomplete to $B_j$ for $j > i$ and
	 is complete to $B_j$ for $j < i-1$; 
\item the set $C_i$ is anticomplete to $D_j$ for $j > i$ and
	is complete to $D_j$ for $j < i-1$;
	\end{itemize}
	\item for each $i = 1, \ldots, k$, 
\begin{itemize}
\item the set $A_i$ is complete to $D_j$ for $j < i$, and is anticomplete to
	$D_j$ for $j \geq i$; 
\item	the set $C_i$ is complete to $B_j$ for $j < i$, and is anticomplete to $B_j$ for $j \geq i$.
\end{itemize}
\end{itemize}
We denote $A=A_1\cup\ldots\cup A_k$, $B=B_1\cup\ldots\cup B_k$, $C=C_1\cup\ldots\cup C_k$, $D=D_1\cup\ldots\cup D_k$ and refer to  a $k$-\textit{\decomp decomposition} of $G$ 
either  as $[(A_1, \ldots, A_k)(B_1, \ldots, B_k) (C_1, \ldots, C_k) (D_1, \ldots, D_k)]$ (long-term notation) or, if no confusion arises, as $(A,B,C,D)$ (short-term notation).

We call the subgraphs $G[A\cup B]$ and $G[C\cup D]$ the \textit{components} of the decomposition.
Let us observe that the decomposition is symmetric with respect to their components, i.e. 
if $(A,B,C,D)$ is a chain decomposition, then $(C,D,A,B)$ also is a chain decomposition of $G$.
However, it is not necessarily symmetric with respect to $U$ and $W$.  In the definition of chain decomposition, we fix both the bipartition of $G$ and the order of its parts. 
By changing the order, we may obtain another chain decomposition of $G$, and we refer to these two decompositions as {\it left} and {\it right}, respectively.

We will say that $G=(U,W,E)$ admits a \decomp decomposition if it admits a 
$k$-\decomp decomposition (left or right) for
some natural $k>0$. Note that if $G$ admits a $1$--\decomp decomposition 
$[(A_1)(B_1)(C_1)(D_1)]$, then $G$ is the disjoint union of $G[A_1 \cup B_1]$ and $G[C_1 \cup D_1]$.
In particular, $G$ is disconnected.

\begin{lemma}\label{lem:reconstruct}
Let $G_1=G[A\cup B]$ and $G_2=G[C\cup D]$ be two components of  a $k$-\decomp decomposition	of $G$. 
Then $G$ can be reconstructed from $G_1$ and $G_2$ with the help of $k$, $A_1$ and $C_1$.
\end{lemma}
\begin{proof}
To reconstruct $G$, we need to determine the adjacencies between the vertices of $G_1$ and $G_2$.
For this, we need to properly partition each of the sets $A,B,C,D$ into (at most) $k$ subsets.
	We will show how to find the partitions of $A$ and $B$; the partitions of $C$ and $D$ can be found 
	in a similar way. 

By definition, every vertex in $B_1$ has a neighbour in $A_1$, and
	$A_1$ is anticomplete to $B_j$ for $j > 1$. Therefore, $B_1 = N_{G_1}(A_1)$, and if $k=1$, then
	we are done. Assume now that $k>1$ and we have identified $A_1, B_1, \ldots, A_{i-1}, B_{i-1}$.
 If $i=k$, then $A_i = A - (A_1\cup\ldots\cup A_{i-1})$
	and $B_i = B - (B_1\cup\ldots\cup B_{i-1})$. If $2 \leq i \leq k-1$, then by definition
	every vertex in $A_i$ has a non-neighbour in $B_{i-1}$, while for $j > i$ the set $A_j$ is complete to $B_{i-1}$. 
Therefore, $A_i$ is the set of vertices in 
	$A - (A_1\cup\ldots\cup A_{i-1})$ that have at least one non-neighbour in $B_{i-1}$.
	Also, every vertex in $B_i$ has a neighbour in $A_i$, while for $j > i$ the set $B_j$ is anti-complete to $A_i$.
Therefore, $B_i = N_{G_1}(A_i) - (B_1\cup\ldots\cup B_{i-1})$.
\end{proof}


\subsection{Decomposition scheme}
\label{sec:dec_scheme}

Let $G=(U,W,E)$ be a bipartite graph and $\overline{G^b}$ the bipartite complement of $G$.
If $G$ or $\overline{G^b}$ admits a chain decomposition, we split the graph into two decomposition 
components and proceed with the components recursively. If this process can decompose $G$ into one-vertex graphs, 
we call $G$ {\it chain-decomposable} or {\it totally decomposable by chain decomposition}. 

If $G$ is chain-decomposable, the recursive procedure for decomposing $G$ into single vertices can be described by a rooted binary decomposition tree $\mathcal{T}(G)$.
Below we describe the rules to construct the tree. To simplify the description, we distinguish between
1-chain decomposition (in which case the graph is disconnected) and $k$-chain decomposition for $k\ge 2$. In the second case, 
we introduce two \textit{marker vertices} $v_1$ and $v_2$ that are needed to properly reconstruct the graph from its two decomposition components.

\begin{enumerate}
	\item If $G$ has only one vertex $v$, then $\mathcal{T}(G)$ consists of one node marked by $(v,f)$, 
where $f$ is a binary flag showing which of the parts $U$ and $W$ vertex $v$ belongs to.

	\item If $G$ is disconnected, then $G = G_1 + G_2$ for some induced subgraphs $G_1$ and $G_2$ of $G$, and the tree $\mathcal{T}(G)$ consists of a root marked by UNION
	and linked to the roots of $\mathcal{T}(G_1)$ and $\mathcal{T}(G_2)$. 
	
	\item If $H=\overline{G^b}$ is disconnected, then $H = H_1 + H_2$, and $\mathcal{T}(G)$
	consists of a root marked by CO-UNION and linked to the roots of $\mathcal{T}(H_1)$ 
	and $\mathcal{T}(H_2)$.
	
	\item If $G$ admits a $k$-\decomp decomposition 
	$$(A,B,C,D)=[(A_1, \ldots, A_k)(B_1, \ldots, B_k)(C_1, \ldots, C_k)(D_1, \ldots, D_k)]$$ with $k \geq 2$, then
	we let $$G_1 = G[A\cup B\cup \{ v_1 \}] \mbox{ and }
	G_2 = G[C\cup D\cup \{ v_2 \}],$$ where $v_1$ and $v_2$ are arbitrary
	vertices from $D_1$ and $B_1$, respectively. In this case, $\mathcal{T}(G)$ consists of a root
	marked by the tuple $(k,v_1,v_2,\mbox{CHAIN})$ 
	and linked to the roots of $\mathcal{T}(G_1)$ and 
	$\mathcal{T}(G_2)$.

	\item If $H=\overline{G^b}$ admits a $k$-\decomp decomposition 
	$$(A,B,C,D)=[(A_1, \ldots, A_k)(B_1, \ldots, B_k)(C_1, \ldots, C_k)(D_1, \ldots, D_k)]$$ with $k \geq 2$, then
	we let $$H_1 = H[A\cup B\cup  \{ v_1 \}]\mbox{  and }
	H_2 = H[C\cup D\cup \{ v_2 \}],$$ where $v_1$ and $v_2$ are arbitrary
	vertices from $D_1$ and $B_1$, respectively. In this case $\mathcal{T}(G)$ consists of a root
	marked by the tuple $(k,v_1,v_2,\mbox{CO-CHAIN})$ 
	and linked to the roots of 
	$\mathcal{T}(H_1)$ and $\mathcal{T}(H_2)$.
\end{enumerate} 

\begin{proposition}
A chain-decomposable graph $G$ can be reconstructed from its decomposition tree $\mathcal{T}(G)$.
\end{proposition}
\begin{proof}
Each node $X$ of $\mathcal{T}(G)$ corresponds to a subgraph $F_X$ of $G$ induced by the leaves of the subtree of $\mathcal{T}(G)$ 
rooted at $X$. In particular, the root of $\mathcal{T}(G)$ corresponds to $G$. 
We reconstruct $G$ by traversing $\mathcal{T}(G)$ in DFS post-order.

If $X$ is marked by UNION or CO-UNION, the reconstruction of $F_X$ from 
the graphs corresponding to the children of $X$ is obvious.

Now assume that $X$ is marked by $(k, v_1, v_2, \mbox{CHAIN})$ and let $F_1=(U_1,W_1,E_1)$ and $F_2=(U_2,W_2,E_2)$ be the graphs corresponding to the children of $X$. 
Then the two decomposition components of $F_X$ are $F_1 - v_1$ and $F_2 - v_2$\footnote{Technical remark: 
since $v_1$ and $v_2$ belong to both $F_1$ and $F_2$, to distinguish between them, we define $F_1$ to be the graph containing a vertex with the smallest label different from $v_1$ and $v_2$.}.
By definition of chain decomposition, $D_1$ is anti-complete to $A_1$ and complete to $A_j$ for all $j>1$. 
Therefore, $A_1= U_1 - N_{F_1}(v_1)$. Similarly, $C_1=U_2 - N_{F_2}(v_2)$. This information enables us to reconstruct $F_X$ by  Lemma~\ref{lem:reconstruct}.

If $X$ is marked by $(k,v_1,v_2,\mbox{CO-CHAIN})$, the graph $F_X$ can be reconstructed in a similar way.
\end{proof}
\subsection{The number of chain-decomposable graphs} 
\label{sec:numOfGraphs}

In this section, we show that the number of $n$-vertex chain-decomposable bipartite graphs grows as $n^{\Theta(n)}$. 
To obtain an upper bound, we estimate the number of decomposition trees for a chain decomposable graph. 
To this end, we start by estimating the number of nodes in these trees.

\subsubsection{On the number of nodes in decomposition trees}
\label{sec:numOfNodes}

Let $k \geq 0$ be a constant. We say that a rooted binary tree $T$ is a \textit{$k$-decomposition tree} 
of an $n$-element set $A$, if every internal node of $T$ has \textit{exactly} two children
and every node $v$ of $T$ is assigned a subset $S(v) \subseteq A$ in such a way that:
\begin{enumerate}
	\item $S(v) = A$ if and only if $v$ is a root;
	\item $|S(v)| = 1$ if and only if  $v$ is a leaf; 
	\item If $v_1$ and $v_2$ are children of $v$, then
	\begin{enumerate}
		\item $S(v_1) \cup S(v_2) = S(v)$;
		\item $|S(v)| \leq |S(v_1)| + |S(v_2)| \leq |S(v)| + k$;
		\item each of $S(v_1)$ and $S(v_2)$ has a private element, i.e. $S(v_1)-S(v_2)$ and $S(v_2)-S(v_1)$ are both non-empty. 
	\end{enumerate}
\end{enumerate}

The next lemma provides an upper bound on the number of \textit{leaves} in a $k$-decomposition tree
of an $n$-element set.

\begin{lemma}\label{lem:numLeaves}
	Let $t(n)$ be the maximum number of leaves in a $k$-decomposition tree of an $n$-element set.
	Then
	$$
		t(n) \leq 
		\begin{cases}
			2^{n-1}, & n \leq k; \\
			(n-k) 2^{k}, & n \geq k+1.
		\end{cases}
	$$
\end{lemma}
\begin{proof}
	Let $T$ be an arbitrary $k$-decomposition tree of an $n$-element set and $r$ be 
	the root of $T$
	with two children $v_1$ and $v_2$. Let us also denote $n_i = |S(v_i)|$ for $i=1,2$.
	We prove the statement by induction on $n$ using the trivial relation $t(n) \leq t(n_1) + t(n_2)$ and
	the basis $t(2) = 2$, which easily follows from the definition of $k$-decomposition tree.
	
	Fist, assume that $n \leq k$. Note that $n_1 \leq n-1$ and $n_2 \leq n-1$, as each of the sets 
	$S(v_1)$ and $S(v_2)$ has a private element. Then by induction 
	$t(n) \leq 2^{n_1-1} + 2^{n_2-1} \leq 2^{n-2} + 2^{n-2} = 2^{n-1}$.
	Now, let $n \geq k+1$. We have to analyze three cases:
	\begin{enumerate}
		\item $n_1 \leq k$ and $n_2 \leq k$. Then by induction 
		$$
			t(n) \leq 2^{n_1-1} + 2^{n_2-1} \leq 2^{k-1} + 2^{k-1} = 2^{k} \leq (n-k) 2^k.
		$$
		
		\item $n_1 \leq k$ and $n_2 \geq k+1$. Then by induction 
		$$
			t(n) \leq 2^{n_1-1} + (n_2-k) 2^{k} \leq 2^{k-1} + (n-1-k)2^{k} = (n - 1/2 - k)2^{k} \leq (n-k)2^k.
		$$
		
		\item $n_1 \geq k+1$ and $n_2 \geq k+1$. Then by induction
		$$
			t(n) \leq (n_1-k) 2^{k} + (n_2-k) 2^{k} = (n_1 + n_2 - 2k)2^k \leq (n-k)2^k,
		$$
		where the latter inequality follows from the fact that $n_1 + n_2 \leq n+k$ by the definition.
	\end{enumerate}
\end{proof}

\begin{corollary}\label{cor:k-decTrees}
	Let $T$ be a $k$-decomposition tree of an $n$-element set. Then $T$ has at most
	$2^n - 1$ nodes, if $n \leq k$, and at most $(n-k)2^{k+1}-1$ nodes, if $n \geq k+1$.
\end{corollary}
\begin{proof}
	The corollary follows from Lemma~\ref{lem:numLeaves} and the fact that a tree with $s$ leaves,
	in which every internal node has at least two children, has at most $2s-1$ nodes.
\end{proof}

In order to use Corollary~\ref{cor:k-decTrees} for estimating the number of nodes in a decomposition
tree $\mathcal{T}(G)$ of an $n$-vertex chain-decomposable graph $G$, it suffices to observe that $\mathcal{T}(G)$ is a 2-decomposition tree of $V(G)$.

\begin{theorem}\label{th:decTreeNodes}
	Let $G$ be an $n$-vertex chain-decomposable bipartite graph with $n \geq 3$. 
	Then $\mathcal{T}(G)$ has at most $8n-17$ nodes.
\end{theorem}


\subsubsection{On the number of chain-decomposable bipartite graphs}

\begin{lemma}\label{lem:numOfGraphs}
	There are at most $n^{O(n)}$ $n$-vertex labeled chain-decomposable bipartite graphs.
\end{lemma}
\begin{proof}
	To prove the lemma, we use the fact that for $n \geq 3$ any decomposition tree of an $n$-vertex 
	chain-decomposable bipartite graph has at most $N=8n-17$ nodes (Theorem~\ref{th:decTreeNodes}).
	
	Let $T$ be a rooted binary tree with at most $N$ nodes. Let us estimate how many different
	decomposition trees of an $n$-vertex graph one can obtain from $T$ by marking its nodes.
	According to the decomposition rules in Section~\ref{sec:dec_scheme} there are at most $2n$
	possibilities to label each of the leaves, 
	and at most $2n^3 +2$ possible labels for each of the internal nodes of $T$.
	All in all there are no more than $(2n^3 + 2n + 2)^{|V(T)|}$ ways to mark the tree $T$.
	Finally, as there are exactly $r^{r-1}$ rooted trees on $r$ vertices, the number of different 
	decomposition trees, and hence the number of $n$-vertex labeled chain-decomposable
	bipartite graphs, is at most
	$$
		\sum\limits_{r=1}^{N} r^{r-1} (2n^3 + 2n + 2)^r \leq N^N (2n^3 + 2n + 2)^N = n^{O(n)}.
	$$
\end{proof}

This lemma provides a factorial upper bound for the number of labeled chain-decom\-po\-sable bipartite graphs. A factorial lower bound 
follows from the obvious fact that all graphs of degree at most 1 (which form one of the minimal factorial classes of graphs) are chain-decomposable. 
Summarizing, we obtain the following conclusion.

\begin{theorem}\label{th:numb_chain_dec}
	There are $n^{\Theta(n)}$ labeled chain-decomposable bipartite graphs on $n$ vertices.
\end{theorem}

\section{$P_7$-free bipartite graphs are chain-decomposable}
\label{sec:P7}

We prove the main result of this section trough a series of technical lemmas. The first of them provides a sufficient condition for a $P_7$-free bipartite graphs 
to admit a chain decomposition. 

\begin{lemma}\label{lem:P7_chain_dec}
	Let $G=(U,W,E)$ be a $P_7$-free bipartite graph such that
	\begin{enumerate}
		\item $A \cup Q \cup C$ is a partition of $U$;
		\item $B \cup R \cup D$ is a partition of $W$;
		\item every vertex in $B$ has a neighbour in $A$;
		\item every vertex in $R$ has a neighbour in $Q$;
		\item every vertex in $D$ has a neighbour in $C$;
		\item $A$ is anticomplete to $R \cup D$;
		\item $C$ is anticomplete to $B \cup R$;
		\item every vertex in $Q$ is complete to at least one of the sets $B$ and $D$.
	\end{enumerate}
	
	\noindent
	Then $G$ admits a \decomp decomposition 
	$$[(A_1, \ldots, A_k)(B_1, \ldots, B_k)(C_1, \ldots, C_k)(D_1, \ldots, D_k)],$$ with
	$A_1 = A$, $B_1 = B$, $C_1 = C$ and $D_1 = D$.
\end{lemma}
\begin{proof}
	The proof is by induction on the size of $Q$. If $Q$ is empty, then $R$ is empty too (see assumption 4). 
Therefore, in this case $[(A)(B)(C)(D)]$ is a 1-\decomp decomposition of $G$.
	
	Assume now that $Q \neq \emptyset$. We partition $Q$ into three sets:
	\begin{enumerate}
		\item[-] $Q_B$ is the set of vertices that are complete to $B$, but have at least one non-neighbour
		in $D$;
		\item[-] $Q_D$ is the set of vertices that are complete to $D$, but have at least one non-neighbour
		in $B$;
		\item[-] $Q_{BD}$ is the set of vertices that are complete to both $B$ and $D$.
	\end{enumerate}
	
	\noindent
	\textbf{Claim 1.} \textit{No vertex in $R$ has neighbours in both $Q_B$ and $Q_D$}.
	
	\medskip
	\noindent
	\textit{Proof.} Suppose to the contrary that $x \in R$ has a neighbour $q_1 \in Q_B$ and a
	neighbour $q_2 \in Q_D$. Then $a, b, q_1, x, q_2, d, c$ induce forbidden $P_7$, where 
	$b \in B$ is a non-neighbour of $q_2$; $a \in A$ is a neighbour of $b$;
	$d \in D$ is a non-neighbour of $q_1$, and $c \in C$ is a neighbour of $d$.
	
\medskip
\medskip
	\noindent
	Claim 1 allows us to partition the set $R$ into three subsets:
	\begin{enumerate}
		\item[-] $R_B$ is the set of vertices that have at least one neighbour in $Q_B$;
		\item[-] $R_D$ is the set of vertices that have at least one neighbour in $Q_D$;
		\item[-] $R_{BD} = R \setminus (R_B \cup R_D)$; note that
		every vertex in $R_{BD}$ has at least one neighbour in $Q_{BD}$.
	\end{enumerate}
	
	\noindent
	\textbf{Claim 2.} \textit{Every vertex in $Q_{BD}$ is complete to at least one of the sets $R_B$ 
	and $R_D$}.
	
	\medskip
	\noindent
	\textit{Proof.} Suppose to the contrary that $x \in Q_{BD}$ has a non-neighbour $r_1 \in R_B$ and 
	a non-neighbour $r_2 \in R_D$. Then $r_1, q_1, b, x, d, q_2, r_2$ induce forbidden $P_7$,
	where $q_1 \in Q_B$ is a neighbour of $r_1$; $d \in D$ is a non-neighbour of $q_1$;
	$q_2 \in Q_D$ is a neighbour of $r_2$, and $b \in B$ is a non-neighbour of $q_2$.
	
	\medskip
	\medskip
	\noindent
	Now we split the analysis into the following two cases:
	\begin{enumerate}
		\item At least one of the sets $R_B$ and $R_D$ is empty. If $R_B = \emptyset$, then
		$$
			[(A, Q_D \cup Q_{BD}) (B, R_D \cup R_{BD}) (C, Q_B) (D, \emptyset)]
		$$ 
		is a 2-\decomp decomposition of $G$.
		Similarly, if $R_D = \emptyset$, then 
		$$
			[(A,Q_D) (B, \emptyset) (C, Q_B \cup Q_{BD}) (D,R_B \cup R_{BD})]
		$$ 
		is a 2-\decomp decomposition of $G$.
		
		\item Both sets $R_B$ and $R_D$ are non-empty. In this case the graph $G' = G[Q \cup R]$
		satisfies the assumptions $1-8$ of the lemma with $Q_D \cup Q_{BD} \cup Q_B$ being a
		partition of $Q$ and $R_D \cup R_{BD} \cup R_B$ being a partition of $R$, and 
		$|Q_{BD}| < |Q|$. Therefore, by induction hypothesis $G'$ admits a $k$-\decomp decomposition
		for some natural $k$
		$$
			[(A_1, \ldots, A_k)(B_1, \ldots, B_k)(C_1, \ldots, C_k)(D_1, \ldots, D_k)],
		$$
		where $A_1 = Q_D$, $B_1 = R_D$, $C_1 = Q_B$, $D_1 = R_B$.
		Now it is easy to check that 
		$$
			[(A,A_1, \ldots, A_k)(B,B_1, \ldots, B_k)(C,C_1, \ldots, C_k)(D,D_1, \ldots, D_k)]
		$$
		is in turn a $(k+1)$-\decomp decomposition of $G$.
	\end{enumerate}
\end{proof}

Given a bipartite graph $G=(U,W,E)$, we associate with each part $C \in \{ U,W \}$ of $G$  the \textit{neighbourhood graph} $G_{C}$ defined as follows:
$G_C = (C,E_C)$, where $E_C = \{ (x,y) | N_G(x) \cap N_G(y) \neq \emptyset \}$. The following lemma characterizes the graphs $G_U$ and $G_W$ in terms of 
two forbidden induced subgraphs, $P_4$ and $square$, where $square$ is a chordless cycle on 4 vertices.

\begin{lemma}
	If $G = (U,W,E)$ is a $P_7$-free bipartite graphs, then both $G_U$ and $G_W$ are 
	$(P_4,square)$-free graphs.
\end{lemma}
\begin{proof}
	It can be easily checked that if, say, $G_U$ contains $P_4$ or $square$ as an induced
	subgraph, then $G$ contains an induced $P_7$.
\end{proof}

$(P_4,square)$-free graphs are known in the literature as \textit{quasi-threshold} \cite{QT} or \textit{trivially perfect} \cite{Golumbic}. 
It is known that a quasi-threshold graph can be constructed from $K_1$ by repeatedly taking either the disjoint 
union of two quasi-threshold graphs or the join of a quasi-threshold graph with $K_1$ .

\begin{lemma}\label{claim:ccs}
	Let $G=(V,E)$ be a connected non-complete quasi-threshold graph. Then
	$G$ has a clique cutset $C$ such that $C$ is complete to $V \setminus C$.
\end{lemma}
\begin{proof}
	We prove the lemma by induction on the number of vertices in $G$. There are no
	connected non-complete graphs with at most 2 vertices. The only  connected 
	non-complete graph with 3 vertices is $P_3$. Clearly, 
	the statement is true for this graph.

	Let $n=|V(G)|> 3$. Assume the statement is true for every graph with fewer than 
	$n$ vertices. Since $G$ is connected, it is the join of a quasi-threshold graph $G'$ with a vertex, say, 
	$c_1 \in V$. If $G'$ is disconnected, then $\{c_1\}$ is the 
	desired clique cutset. Otherwise $G'$ is a connected non-complete quasi-threshold
	graph, and hence by induction $G'$ contains a clique cutset $C'$, which is complete to 
	$V(G') \setminus C'$. Clearly, $C' \cup \{c_1\}$ is the desired clique cutset in $G$.
\end{proof}

\begin{lemma}\label{lem:P7chainDecomposition}
	Let $G=(U,W,E)$ be a connected $P_7$-free bipartite graph. 
	If $G_U$ (resp. $G_W$) is non-complete, then $G$ admits 
	a left (resp. right) \decomp decomposition.
\end{lemma}
\begin{proof}
	We prove the statement for $G_U$. For $G_W$, the proof is similar.
	
	If $G=(U,W,E)$ is connected, the so is $G_U$. Therefore, by Lemma~\ref{claim:ccs}, $G_U$ contains a clique cutset $Q$ which is complete to 
	$U \setminus Q$. Let $G_U[A_1], \ldots, G_U[A_k]$, $k \geq 2$, be 
	connected components of $G_U \setminus Q$. These components 
	correspond to connected components of $G \setminus Q$. For each $i$, we denote by $F_i = (A_i,B_i,E_i)$  the connected component 
	of $G \setminus Q$ that contains $A_i$. 
	Since $G$ is connected, every vertex in $W$ has a neighbour in $U$. Therefore,
	the vertices of $G$ are partitioned as follows: $U = A_1 \cup \ldots \cup A_k \cup Q$ and 
	$W = B_1 \cup \ldots \cup B_k \cup R$, where 
	$R = N(Q) \setminus (B_1 \cup \ldots \cup B_k)$. 
	
	Let $x$ be a vertex in $Q$. Since $x$ is a dominating vertex in $G_U$, every other
	vertex in $U$ has a common neighbour with $x$ in $G$. In particular, $x$ has a 
	neighbour in each of $B_i$, $i = 1, \ldots, k$. 
	We claim that $x$ is complete to all but at most one set $B_i$, $i = 1, \ldots, k$. 
	Indeed, if $x$ has a non-neighbour $y \in B_i$ and a non-neighbour $z \in B_j$, 
	$i \neq j$, then a shortest path between $x$ and $y$, and a shortest path 
	between $x$ and $z$ together form a path which contains an induced $P_7$.
	
	To complete the proof, we denote $A=A_1$, $B=B_1$, $C=\bigcup_{i=2}^k A_i$, $D=\bigcup_{i=2}^k B_i$ and observe that
the sets $A,Q,C,B,R,D$ satisfy the assumptions of Lemma~\ref{lem:P7_chain_dec}. Therefore,
	$G$ admits a \decomp decomposition.
\end{proof}

\begin{lemma}\label{lem:neighGraph}
	Let $G=(U,W,E)$ be a $P_7$-free bipartite graph with at least three vertices,
	and $H=\overline{G^b}$. Then at least one of the following graphs is not 
	complete: $G_U, G_W, H_U$.
\end{lemma}
\begin{proof}
	If $G_U$ or $G_W$ is not complete, then we are done. 
	Assume $G_U$ and $G_W$ are complete, i.e any two vertices in the same 
	color class have a common neighbour in $G$. This implies, in particular,  that $G$ 
	is connected.
	
	Let $x$ and $y$ be two vertices in $U$ such that the union of their neighbourhoods
	is of maximum cardinality, that is 
	$$
		|N_G(x) \cup N_G(y)| = \max\limits_{a,b \in U} |N_G(a) \cup N_G(b)|.
	$$
	Denote $S = N_G(x) \cup N_G(y)$, $S_x = N_G(x) \setminus N_G(y)$,
	$S_y = N_G(y) \setminus N_G(x)$, $S_{xy} = N_G(x) \cap N_G(y)$ and
	$\overline{S} = W \setminus S$.
	
	 Assume $\overline{S} \neq \emptyset$. This implies that 
\begin{itemize}
\item[(1)] {\it $N_G(x)$ and $N_G(y)$ are incomparable}, i.e. neither of them contains the other.
Indeed, suppose to the contrary that
	$N_G(y) \subseteq N_G(x)$, and let $s$ be a vertex in $\overline{S}$. Since $G$ is
	connected, $s$ has a neighbour $z$ in $U$. But then 
	$|N_G(x) \cup N_G(z)| > |S|$, contradicting the choice of $x,y$. 
	Therefore both $S_x$ and $S_y$ are non-empty. 
\item[(2)] {\it every vertex of $U$ that has a neighbour in $\overline{S}$ is anticomplete either to $S_x$ or to $S_y$}.	
To show this, consider a vertex $z \in U$ that has a neighbour in $\overline{S}$. 
If $z$ is complete to $S_x$, then $|N_G(y) \cup N_G(z)| > |S|$. This contradicts the choice of $x,y$ and proves that 
$z$ has a non-neighbour in $S_x$. Similarly, $z$ has a non-neighbour in $S_y$.
If, in addition, $z$ has neighbours in both $S_x$ and $S_y$, then $x,y,z$ together with a neighbour and a non-neighbour of $z$ in $S_x$ and 
a neighbour and a non-neighbour of $z$ in $S_x$ induce a $P_7$. This contradiction shows that $z$
is anticomplete either to $S_x$ or to $S_y$. 
\end{itemize}		
	Now let $v \in \overline{S}$, $x_1 \in S_x$ and $y_1 \in S_y$. Since $G_W$ is
	complete, $v,x_1$ have a common neighbour $x_2$, and $v,y_1$ have a common
	neighbour $y_2$. By (2),  $x_2 \neq y_2$, and $x_2$ is not adjacent to $y_1$, and
	$y_2$ is not adjacent to $x_1$. But then $\{ x,x_1,x_2,v,y_2,y_1,y \}$ induces a $P_7$. 
	This contradiction shows that $S = W$. But then $N_H(x) \cap N_H(y) = \emptyset$, and therefore $H_U$ is not
	complete.
\end{proof}

\begin{theorem}\label{thm:P7_chain_dec}
Every $P_7$-free bipartite graph is chain-decomposable. 
\end{theorem} 

\begin{proof}
Let $G=(U,W,E)$ be a $P_7$-free bipartite graph and $H=\overline{G^b}$.
If $G$ or $H$ is disconnected, then we apply Rules 2 or 3 of the chain decomposition scheme. 

Assume now that both $G$ are $H$ are connected. By Lemma~\ref{lem:neighGraph} at least one of the three graphs $G_U, G_W, H_U$
is not complete and hence by Lemma~\ref{lem:P7chainDecomposition}, $G$ or $H$ admits 
	a \decomp decomposition, in which case we apply Rules 4 or 5.

Since the class of $P_7$-free bipartite graphs is hereditary, repeated applications of the above procedure decompose $G$ into single vertices, 
i.e. $G$ is chain-decomposable.
\end{proof}

\noindent
Theorems \ref{th:numb_chain_dec} and \ref{thm:P7_chain_dec} imply the main result of the paper.
\begin{theorem}\label{thm:main}
The class of $P_7$-free bipartite graphs is factorial.
\end{theorem}

\section{Concluding remarks and open problems}

In this paper, we complete the description of speeds of monogenic classes of bipartite graphs by showing that the class of $P_7$-free bipartite graphs is factorial. 
This result answers an open question from \cite{Peter} and several related open questions from \cite{LBC, Zam11, Zam13}. 
In particular, our result implies that the class of  $(K_t,P_7)$-free graphs is factorial for 
any value of $t\ge 3$.

\begin{theorem}
For any $t\ge 3$, the class of $(K_t,P_7)$-free graphs is factorial.
\end{theorem}

\begin{proof}
The class of $(K_t,P_7)$-free graphs contains all bipartite  $P_7$-free graphs and hence is at least factorial. For an upper factorial bound, we use the notion of 
locally bounded coverings introduced in \cite{LBC} and the idea of $\chi$-bounded classes. It is known (see e.g. \cite{Gyarfas}) that $P_k$-free graphs are 
$\chi$-bounded for any value of $k$ and hence the chromatic number of $(K_t,P_7)$-free graphs is bounded by a constant. As a result, every $(K_t,P_7)$-free graph $G$
can be covered by $P_7$-free bipartite graphs in such a way that every vertex of $G$ is covered by finitely many such graphs. Together with Theorem~\ref{thm:main} and
the results from \cite{LBC} this implies a factorial upper bound for $(K_t,P_7)$-free graphs.
\end{proof}

In spite of the progress made in this paper, the problem of characterizing the factorial layer remains widely open.
In the introduction, we mentioned a conjecture reducing this problem to hereditary properties of bipartite graphs.
The restriction to bipartite graphs is important on its own right and the problem remains quite challenging even under this restriction. 
In the present paper, we have completed a solution to the problem of characterizing factorial classes of bipartite graphs defined by a {\it single} forbidden induced subgraph. 
The obvious next step is characterizing factorial classes of bipartite graphs defined by {\it finitely many forbidden induced bipartite subgraphs}.

Let $M$ be a finite set of bipartite graphs and $X_M$ the class of $M$-free bipartite graphs.
In \cite{GT}, we have shown that $X_M$ is factorial if and only if it contains no boundary class.
This idea can be roughly explained as follows. It is known (can be derived  e.g. from the results in \cite{LUW95}) that the class of $(C_3,C_4,\ldots,C_k)$-free bipartite graphs 
is superfactorial for any fixed value of $k$. With $k$ tending to infinity, this sequence converges to the class of forests.
Therefore, if $X_M$ contains the class of forests, then $X_M$ is superfactorial. Indeed, in this case every graph in $M$ contains a cycle and hence $X_M$ contains the 
class of $(C_3,C_4,\ldots,C_k)$-free bipartite graphs, where $k$ is the size of a largest cycle in graphs in $M$.  
Similarly,  if $X_M$ contains the class of bipartite complements of forests, then $X_M$ is superfactorial.
Therefore, a necessary condition for $X_M$ to be factorial is that $M$ contains a forest and the bipartite complement of a forest.
The results in \cite{Peter} together with the main results of this work show that this condition is also
sufficient when $M$ consists of a single graph. 
We believe that the condition remains sufficient for an arbitrary set $M$. To prove this, it would be enough to
settle the following conjecture.

\begin{conjecture}\label{conj:tree}
	For any tree $T$, the class of $\{ T, \overline{T^b} \}$-free graphs is at most factorial.
\end{conjecture}

\noindent
In the terminology of boundary classes, this conjecture is equivalent to saying that in the family of hereditary properties of bipartite graphs 
the class of forests and the class of their bipartite complements are the only boundary classes.

An important special case of Conjecture~\ref{conj:tree} deals with $\{ P_k, \overline{P_k^b} \}$-free bipartite graphs.
In the present paper, we solved this case for $k \leq 7$.

\medskip
To prove our main result we showed that every $P_7$-free bipartite graph is totally decomposable
with respect to chain decomposition and bipartite complementation. It is not difficult to see that 
the class of $P_7$-free bipartite graphs is a proper subclass of chain decomposable graphs.
For instance, every path is chain-decomposable. 
It would be interesting to understand which graphs are chain-decomposable. 
In particular, is the class of chain-decomposable graphs \textit{hereditary}? 
If yes, what are the minimal graphs which are not chain-decomposable?

\end{document}